\documentclass[a4paper,12pt]{article}
\usepackage{bm}
\usepackage{amsmath}
\usepackage{amsthm}
\usepackage{wasysym}
\usepackage{amssymb}
\usepackage{amsfonts}
\usepackage{graphicx}
\usepackage[english]{babel}
\usepackage[utf8]{inputenc}
\usepackage[T1]{fontenc}
\usepackage{mathrsfs}
\usepackage{enumerate}
\usepackage{version}
\usepackage{calc}
\usepackage{subfigure}
\usepackage{authblk}
\usepackage{comment}
 \usepackage{pstricks,pst-math,pst-xkey}
 \usepackage{float}
 \usepackage{indentfirst}
\usepackage[pagewise]{lineno} 
\newtheorem{definition}{Definition}[section]
\newtheorem{theorem}[definition]{Theorem}
\newtheorem{lemma}[definition]{Lemma}
\newtheorem{corollary}[definition]{Corollary}

\numberwithin{equation}{section}

\newcommand{\N}{{\mathbb N}}

\newcommand{\R}{{\mathbb R}}

\newcommand{\ve}{{\varepsilon}}
\newcommand{\wwtilde}[1]{\widetilde{\widetilde{#1}}}

\def\ve{\varepsilon}

\def\f{\varphi}

\def\d{\partial}

\def\N{\mathbb{N}}

\def\f{\varphi}

\def\R{\mathbb{R}}
\def\E{\mathbb{E}}
\def\N{\mathbb{N}}

\def\Lf{\Lambda_f}

\begin{document}

\title{On the Optimal Control Problem of Stochastic Semilinear Partial Differential Equations with Non-Globally Lipschitz Coefficients}
\date{}

\author[1]{Oleksiy Kapustyan \thanks{kapustyan@knu.edu}}

\author[3] {Olha Martynyuk
\thanks{o.martynyuk@chnu.edu.ua}}

\author[2]{Oleksandr Misiats\thanks{omisiats@vcu.edu}}

\author[1]{Oleksandr Stanzhytskyi \thanks{stanzhytskyi@knu.edu}}

\affil[1]{Department of Mathematics,
Taras Shevchenko National University of Kyiv, Ukraine}

\affil[2]{Department of Mathematics, Virginia Commonwealth University,
Richmond, VA, 23284, USA}

\affil[3]{Department of Mathematics, Yuriy Fedkovych Chernivtsi National University, Chernivtsi, Ukraine}

\maketitle

\begin{abstract}
In this paper, we study optimal control problems for stochastic
semilinear partial differential equations, which lack the maximum principle, and whose coefficients do not have bounded Frechet derivatives. We propose an approximation scheme
for the corresponding optimization problem, and prove convergence
of the approximating solutions on both finite and infinite
time intervals. 

\medskip

\end{abstract}

\medskip


\section{Introduction}

In this paper, we study the optimal control problem for stochastic
semilinear partial differential equations

\begin{equation}\label{1.1}
\begin{cases}
\begin{aligned}
& dy = \bigl(Ay + f(y) + u(t)\bigr)\,dt
      + \sigma(y)\,dW(t), \\
& y(t,x) = 0,
\qquad x\in \partial D,\; t\in(0,T), \\
& y(0,x) = y_0(x,\omega).
\end{aligned}
\end{cases}
\end{equation}
with the cost function
\begin{equation}\label{1.2}
J(u)
=
\E\int_{0}^{T}\int_{D} y^2(t,x)\,dx\,dt
+
\E\int_{0}^{T}\int_{D} u^2(t,x)\,dx\,dt
\;\to\;
\inf .
\end{equation}
on a finite time interval, as well as
\begin{equation}\label{1.3}
J(u)
=
\E\int_{0}^{\infty}\int_{D} e^{-\gamma t} y^2(t,x)\,dx\,dt
+
\E\int_{0}^{\infty}\int_{D} u^2(t,x)\,dx\,dt
\to \inf 
\end{equation}
on the unbounded interval $t \in [0,+\infty)$. 
Here $D\subset \mathbb \R^d$, $d\ge 1$, is a bounded domain with a
sufficiently regular boundary (for example, satisfying the Lyapunov
condition), and $A$ is a symmetric elliptic operator

\[
Ay=A(x)y
=
\sum_{i,j=1}^{d}
\partial_{x_i}
\Bigl(
a_{ij}(x)\partial_{x_j}y
\Bigr)
+
a_0(x)y
=
A_1y+a_0(x)y,
\]
with the coefficients
$a_{ij},a_0\in L^\infty(D)$,
$a_{ij}=a_{ji}$,
satisfying, for some $\Lambda_A>0$, the ellipticity condition
\begin{equation}\label{1.4}
\sum_{i,j=1}^{d}
a_{ij}(x)\eta_i\eta_j
\ge
\Lambda_A |\eta|^2,
\qquad
\forall\,\eta\in\mathbb \R^d,
\ \text{for a.e. } x\in D.
\end{equation}
Let
$W(t)$ be an infinite-dimensional $Q$-Wiener process with values in
$L^2(D)$.
As for the nonlinearities $f,\sigma:\mathbb R\to\mathbb R$, we assume that $\sigma$ is globally Lipschitz,
$f\in C^1(\mathbb R)$,
$f(0)=0$, and
\begin{equation}\label{1.5}
f'(s)\le \lambda_f,
\end{equation}
for some nonnegative constant $\lambda_f \geq 0$.

The stochastic process
$u=u(t,x,\omega)$
is regarded as the control variable.
We assume that
$
u(t,x,\omega)\in \mathbb R,
$
and
\[
\E\int_0^T\int_D u^2(t,x)\,dx\,dt < \infty
\]
for the problem \eqref{1.2}, while
\[
\E\int_0^\infty\int_D u^2(t,x)\,dx\,dt < \infty
\]
for the problem \eqref{1.3}.

The well-posedness questions for problems of the form \eqref{1.1}-\eqref{1.2} and \eqref{1.1}-\eqref{1.3}, namely, the existence of
optimal pairs $(y^*(t,\omega),u^*(t,\omega))$, have been studied in
many works; see, for example, \cite{1,3}, where conditions ensuring
existence were obtained under the assumptions of monotonicity and polynomial growth of the
nonlinearity $f(y)$. In the works \cite{4,5,6} et. al. and related papers, the existence
of optimal controls was established using the theory of backward
stochastic differential equations (BSDEs). The methods for constructing
optimal controls in the feedback form were also developed there. However,
these results require the coefficients to be globally Lipschitz continuous.

In \cite{7,8}, the existence results were obtained for the optimal controls
of dissipative stochastic partial differential equations under the
polynomial growth assumptions on the coefficients. However, in applications, one often encounters situations in which
the reaction term $f$ exhibits faster than polynomial growth.
For example, this occurs in the exponential Frank--Kamenetskii
equation \cite{9}. Therefore, it is important to establish the
existence of optimal controls in such situations as well, assuming
only the condition \eqref{1.5}. This condition is clearly satisfied by the function
$f(s)=e^{-s},$ for which the standard polynomial growth condition fails. The presence of noise in the equation \eqref{1.1} introduces a
certain relation between the functions $f$ and
$\sigma$. More precisely, we require 
\begin{equation}
\bigl|f'(s)\sigma^{2}(s)\bigr|
\le
C\bigl(1+|s|^{2}\bigr),
\label{1.6}
\end{equation}
to hold for some $C>0$. It is clear that if $\sigma(s)$ decreases sufficiently rapidly as
$s\to\infty$, then the function $f(s)=e^{-s}$ satisfies the condition
\eqref{1.6}. Consequently, our existence results for the optimal controls of
equation \eqref{1.1} cover several important classes of problems in
which the reaction term may exceed the polynomial
growth.

Besides proving the existence of an optimal control, an important
problem is the construction of such controls, in particular, by means
of the stochastic maximum principle. For finite-dimensional stochastic systems, the first results in this
direction were obtained by A.~Bensoussan \cite{10}, and were further
developed in numerous subsequent works; see, for example,
\cite{11,12,13,14,15}.
However, in the above-mentioned papers, the application of the
maximum principle requires the coefficients of the equation to possess
bounded derivatives, which significantly restricts its applicability.

In order to study the problems \eqref{1.1}--\eqref{1.2} and \eqref{1.1}--\eqref{1.3},
where $f(y)$ is not globally Lipschitz, we propose the following
approach. Rather than approaching the original problem directly, we consider a family of
auxiliary problems with globally Lipschitz coefficients. To this end, for every $k \geq 1$ define
\[
P_k(s):=\min\{\max(-k,s),k\},
\qquad
f_k(s)=f\bigl(P_k(s)\bigr),
\]
and consider, similarly to \cite{16}, the sequence of ``cut-off'' optimal control problems

\begin{equation}
\left\{
\begin{aligned}
dy_k
&=
\bigl(Ay_k+f_k(y_k)+u\bigr)\,dt
+\sigma(y_k)\,dW(t),
\\[1ex]
J_k(u)
&=
E\int_0^T\int_D y_k^2(t,x)\,dx\,dt
+
E\int_0^T\int_D u^2(t,x)\,dx\,dt
\to \inf .
\end{aligned}
\right.
\label{1.7}
\end{equation}
 Since the coefficients
$f_k'(s)$ and $\sigma'(s)$ are globally Lipschitz and  possess bounded
derivatives (after a suitable smoothing, if necessary),
 the stochastic maximum principle is applicable to
problem \eqref{1.7}. Loosely speaking, the main result of the present paper is the following: if $
(J^*,u^*,y^*)$ and $
(J_k^*,u_k^*,y_k^*)
$
are the solutions of problems
\eqref{1.1}--\eqref{1.2}
and
\eqref{1.7},
respectively, then:

\begin{enumerate}
\item[(a)]
\[
J_k^* \to J^*,
\qquad k\to\infty.
\]

\item[(b)]
along a subsequence, we have the following convergences:
\[
u_k^* \to u^*,
\qquad
y_k^* \to y^*
\quad\text{in }L^2(\Omega_T,H),
\]
where
\[
\Omega_T=\Omega\times[0,T],
\qquad
H=L^2(D).
\]
\end{enumerate}
As for the problem \eqref{1.1}--\eqref{1.3} on the infinite time interval (i.e. infinite horizon), we propose the following approximation scheme. For each natural number $n$, let
\[
(u_n^*,y_n^*,J_n^*)
\]
denote a solution of the optimal control problem on the interval
$[0,n]$. Define
\[
u_{n,\infty}^*(t)
=
\begin{cases}
u_n^*(t), & t\in[0,n],\\[1ex]
0, & t>n.
\end{cases}
\]
Clearly, this control is admissible for the infinite-horizon problem. Loosely speaking, in the present paper we establish the following assertions. If  $(J^*,u^*,y(t,u^*))$ is a solution of the problem \eqref{1.1}--\eqref{1.3}, then
\begin{enumerate}
\item[(a)]
\[
J_n^* \to J^*,
\qquad n\to\infty.
\]
\item[(b)]
Along a subsequence,
\[
J(u_{n_k,\infty}^*)
\to
J^*,
\qquad
n_k\to\infty,
\]
that is, the sequence $u_{n_k,\infty}^*$ is minimizing for the problem
\eqref{1.1}--\eqref{1.3}.
\item[(c)]
\[
u_{n_k,\infty}^*
\rightharpoonup
u^*
\quad\text{weakly in}\quad
L^2\bigl(\Omega\times[0,\infty),H\bigr),
\qquad
n_k\to\infty.
\]
\item[(d)]
\[
y_{n_k}\bigl(t,u_{n_k,\infty}^*\bigr)
\rightharpoonup
y(t,u^*),
\qquad
n_k\to\infty,
\]
weakly in the space, endowed with the norm
\[
\|y\|_\gamma^2:= \E\int_0^\infty
\int_D
y^2(t,x)\,dx\,e^{-\gamma t}\,dt.
\]
\end{enumerate}

This article is structured as follows.
In Section \ref{Sec:2}, we introduce all necessary concepts and formulate the
main results.
Section \ref{Sec:3} is devoted to the proof of several auxiliary lemmas.
In Section \ref{Sec:4}, we prove the main results.

\section{Preliminaries and Main Results}\label{Sec:2}

Throughout this paper, we will work with the following functional spaces
\[
H=L^2(D),
\qquad
V=H_0^1(D),
\qquad
V'=H^{-1}(D).
\]
In this case
\[
V\subset H\subset V'
\]
is called a Gelfand triple. We denote by $\|\cdot\|$ the norm in
$H$, by $(\cdot,\cdot)$ the inner product in $H$, and by
$\langle\cdot,\cdot\rangle$ the duality pairing between $V'$ and $V$.
That is,

\[
\langle z,v\rangle = z(v),
\qquad
z\in V',
\quad
v\in V.
\]
It follows that
\[
\langle z,v\rangle = (z,v)
\]
for all $z\in H$ and $v\in V$. Let
\[
(\Omega,\mathcal{F},(\mathcal{F}_t),P)
\]
be a complete probability space equipped with a normal filtration
$(\mathcal{F}_t)$, $t\in[0,T]$.
We denote
\[
D_T=D\times[0,T].
\]
Let
\[
\lambda_i>0,
\qquad
\sum_{i=1}^{\infty}\lambda_i^2<\infty,
\]
and let $\{e_i\}$ be an orthonormal basis in $H$ such that $e_i\in L^\infty(D)$ and
\[
\sup_i \|e_i\|_{L^\infty(D)}<\infty.
\]
According to \cite{17}, such basis always exists. Next, we introduce the operator $Q\in\mathcal{L}(H)$
such that $Q$ is nonnegative and
\[
Qe_i=\lambda_i^2 e_i,
\qquad
i\ge 1.
\]
Then
\[
\lambda:=\operatorname{Tr}(Q)
=
\sum_{i=1}^{\infty}\lambda_i^2
<
\infty.
\]
We introduce the following $H$-valued stochastic process:
\begin{equation}
\label{2.1}
W(t)
=
\sum_{i=1}^{\infty}
\lambda_i e_i \beta_i(t),
\qquad
t\ge 0,
\end{equation}
which is a $Q$--Wiener process. Here $\beta_i(t)$ are standard real-valued mutually independent
Wiener processes. We also assume that

\begin{enumerate}
\item[(i)]
$W(t)$ is $\mathcal{F}_t$--measurable;

\item[(ii)]
$W(t+h)-W(t)$ is independent of $\mathcal{F}_t$ for all
$h\ge 0$, $t\ge 0$.
\end{enumerate}
Denote
\[
U=Q^{\frac12}(H).
\]
From \cite[Lemma 2.2]{18}, we may conclude that 
$U\subset L^\infty(D)$. Following \cite{18}, we introduce the multiplication
operator
$\Phi:U\to H$ as
\[
\Phi(\psi):=\varphi\psi,
\qquad
\psi\in U,
\]
for every $\varphi\in H$. Since $\psi\in L^\infty(D)$, this operator is
well defined. Hence,
\[
\Phi\circ Q^{\frac12}:H\to H
\]
defines a Hilbert--Schmidt operator. In this paper, $L_2^0 = L^2\!\bigl(Q^{\frac12}H;H\bigr)$
will denote the space of all Hilbert--Schmidt operators endowed with
the norm $\|\cdot\|_{L_2^0}$.
This way
\begin{equation}
\label{2.1}
\begin{aligned}
\|\Phi\circ Q^{\frac12}\|_{L_2^0}^{2}
=
\sum_{i=1}^{\infty}
\bigl\|
\Phi\circ Q^{\frac12}e_i
\bigr\|^{2} =
\sum_{i=1}^{\infty}
\lambda_i^{2}
\int_D
\varphi^{2}(x)e_i^{2}(x)\,dx
\le
\sup_i
\|e_i\|_{L^\infty(D)}^{2}
\,
\|\varphi\|^{2}
\,
\lambda,
\end{aligned}
\end{equation}
hence, $
\Phi:D_T\to\mathcal{L}(U;H)
$ is a predictable process (with respect to $\mathcal{F}_t$)
satisfying
\[
\E\int_0^t
\|\Phi\circ Q^{\frac12}\|_{L_2^0}^{2}
\,ds
<
\infty .
\]
According to \cite{19}, we can define the stochastic integral
\[
\int_0^t \Phi(s)\,dW(s)\in H
\]
via
\[
\int_0^t \Phi(s)\,dW(s):=
\sum_{i=1}^{\infty}
\lambda_i
\int_0^t
\Phi(s, \cdot)e_i(\cdot)\,
d\beta_i(s),
\]
and
\begin{equation}
\label{2.2}
\E\left\|
\int_0^t
\Phi(s)\,dW(s)
\right\|^{2}
\le
\lambda
\sup_i
\|e_i\|_{L^\infty(D)}^{2}
\int_0^t
E\|\Phi(s,\cdot)\|^{2}\,ds .
\end{equation}
We impose the following conditions on $f$ and $\sigma$.

\medskip

\noindent
{\bf (A1)} The function
$f:\mathbb{R}\to\mathbb{R}
$ belongs to $C^1(\mathbb{R})$, satisfies
$f(0)=0,$
and there exists a constant $\lambda_f>0$ such that

\begin{equation}
f'(s)\le \lambda_f,
\qquad
\forall\, s\in\mathbb{R}.
\label{2.3}
\end{equation}

\medskip

\noindent
{\bf (A2)} The function
\[
\sigma:\mathbb{R}\to\mathbb{R}
\]
is globally Lipschitz with Lipschitz constant $L_\sigma>0$.
In particular, the assumption (A2) implies
\[
|\sigma(s)|
\le
L_\sigma |s| + |\sigma(0)|.
\]
\medskip
\noindent
{\bf (A3)} There exists a constant $C>0$ such that
\begin{equation}
|f'(s)\sigma^2(s)|
\le
C\bigl(1+|s|^2\bigr).
\label{2.4}
\end{equation}

\medskip
Let $y_0$ be $\mathcal{F}_0$--measurable and let the random process
$u(t)$ be $\mathcal{F}_t$--measurable.

\begin{definition}\label{Def:2.1}
An $\mathcal{F}_t$--adapted random process

\[
y(t)\in L^2(\Omega_T;V)
\]

is called a weak solution of \eqref{1.1} on $[0,T]$ if for every
$\varphi\in V$ one has

\begin{equation}
\begin{aligned}
(y(t),\varphi)
&=
(y_0,\varphi)
+
\int_0^t
\Bigl(
\langle A_1y(s),\varphi\rangle
+
(a_0y(s),\varphi)
+
(u(s),\varphi)
+
(f(y(s)),\varphi)
\Bigr)\,ds
\\
&\quad
+
\int_0^t
\bigl(
\varphi,\sigma(y(s))\,dW(s)
\bigr),
\qquad
\text{for a.e. } t\in[0,T],
\end{aligned}
\label{2.5}
\end{equation}

with probability $1$.
\end{definition}

The conditions for the existence and uniqueness of a weak
solution of \eqref{1.1} were obtained in \cite{20}. For the sake of completeness, we provide the corresponding theorem.

\begin{theorem}{\cite{20}}\label{Th:2.1}
Assume that conditions {\rm{\bf(A1)--(A3)}} hold,
$
y_0\in L^\infty\!\bigl(\Omega;L^\infty(D)\bigr)
$
is $\mathcal{F}_0$--measurable, and
\[
\E\int_0^T \|u(t)\|^2\,dt < \infty.
\]
Then the equation \eqref{1.1} has a unique weak solution on $[0,T]$,
and the following energy equality holds with probability $1$:
\begin{equation}
\begin{aligned}
\|y(t)\|^2
&=
\|y_0\|^2
+
2\int_0^t
\Bigl(
\langle A_1y(s),y(s)\rangle
+
(a_0y(s),y(s))
+
(u(s),y(s))
\\
&\hspace{2.8cm}
+
(f(y(s)),y(s))
\Bigr)\,ds
+
\int_0^t
\|\sigma(y(s))\|_{L_2^0}^{\,2}\,ds
+
2\int_0^t
\bigl(
\sigma(y(s))\,dW(s),y(s)
\bigr).
\end{aligned}
\label{2.6}
\end{equation}
\end{theorem}
For the solution of problems
\eqref{1.1}--\eqref{1.2}
and
\eqref{1.1}--\eqref{1.3},
we need the following concept of a weak martingale solution.
\begin{definition}
We say that equation \eqref{1.1} has a {\it weak martingale solution} if
there exist a complete probability space,$(\widetilde{\Omega}, \widetilde{\mathcal F}, \widetilde P),$
a filtration
$(\widetilde{\mathcal F}_t),$
a $Q$--Wiener process $\widetilde W(t)$, adapted to $(\widetilde{\mathcal F}_t)$, and an
$\widetilde{\mathcal F}_t$--adapted process
$
\widetilde u(t),
$
such that
\[
\mathcal L(u)
=
\mathcal L(\widetilde u),
\]
exists $
\widetilde y_0\in L^2(\widetilde\Omega;H),
$ adapted to $\widetilde{\mathcal F}_0$ and satisfying
\[
\mathcal L(y_0)
=
\mathcal L(\widetilde y_0),
\]
and exists an
$\widetilde{\mathcal F}_t$--adapted process
\[
\widetilde y
\in
L^2(\widetilde\Omega_T;V),
\]
such that $\widetilde P$--almost surely, for almost every
$t\in[0,T]$, the following weak formulation holds
\begin{equation}
\begin{aligned}
(\widetilde y(t),\varphi)
&=
(\widetilde y_0,\varphi)
\quad
+
\int_0^t
\Bigl(
\langle A_1\widetilde y(s),\varphi\rangle
+
(a_0\widetilde y(s),\varphi)
+
(\widetilde u(s),\varphi)
\\
&\qquad\qquad
+
(f(\widetilde y(s)),\varphi)
\Bigr)\,ds
\quad
+
\int_0^t
\bigl(
\varphi,
\sigma(\widetilde y(s))
\,d\widetilde W(s)
\bigr),
\end{aligned}
\label{2.8}
\end{equation}
for every $\varphi\in V$.
\end{definition}
An $\widetilde{\mathcal F}_t$--adapted process
$
\widetilde u
\in
L^2(\widetilde\Omega_T;H)
$
is called a \emph{weak admissible control}. A weak martingale solution of \eqref{1.1}
corresponding to the weak admissible control
$\widetilde u$
will be denoted by $
\widetilde y(t,\widetilde u).$ The set of all weak admissible controls is denoted by $U_{ad}$.

\begin{definition}\label{Def:2.3}
We say that the pair

\[
\bigl(
\widetilde u^{\,*}(t),
\widetilde y(t,\widetilde u^{\,*})
\bigr)
\]
is a weak martingale solution of the optimal control problem
\eqref{1.1}--\eqref{1.2} if
$\widetilde u^{\,*}$ is a weak admissible control and
$\widetilde y(t,\widetilde u^{\,*})$
is the solution of \eqref{1.1} in the sense of Definition~2.2
corresponding to $\widetilde u^{\,*}$, and

\begin{equation}
J(\widetilde u^{\,*})
\le
J(\widetilde u)
\label{2.9}
\end{equation}
for every weak admissible control $\widetilde u$.
\end{definition}
For the problem \eqref{1.1}--\eqref{1.3}, the notion of solution is
introduced analogously.

We are now in the position to state the main theorems of this paper. Let us first consider the optimal control problem on a finite time interval.

\begin{theorem}\label{Th:2.2}
Assume the conditions \eqref{1.4} and {{\bf \rm(A1)--(A3)}}
hold, and
$y_0\in L^\infty\!\bigl(\Omega;L^\infty(D)\bigr).
$
Then the problems \eqref{1.1}--\eqref{1.2} and \eqref{1.7}
possess weak martingale solution

\[
\bigl(\widetilde u^{\,*},
\widetilde y(\cdot,\widetilde u^{\,*})\bigr) \ \text{ and  } \bigl(\widetilde u_k^{\,*},
\widetilde y_k(\cdot,\widetilde u_k^{\,*})\bigr),
\]
respectively.
\end{theorem}

Denote

\[
J^{*}=J(\widetilde u^{\,*}),
\qquad
J_k^{*}=J(\widetilde u_k^{\,*}).
\]

Since problem \eqref{1.7} satisfies the assumptions of the stochastic
maximum principle (as the function $f_k$ is smooth with bounded derivative), one can proceed with constructing the solution of \eqref{1.7}. The following
theorem provides a method for constructing solutions of problem
\eqref{1.1}--\eqref{1.2} using the solutions of \eqref{1.7} as an approximating sequence:

\begin{theorem}\label{Th:2.3}
Under the assumptions of Theorem \eqref{Th:2.2}, the following
assertions hold:

\begin{enumerate}
\item
\[
J_k^{*}\to J^{*},
\qquad
k\to\infty.
\]

\item
The family $\{\widetilde u_k^{\,*}\}$
is strongly compact in $
L^2(\widetilde\Omega;H)
$,
and every weak subsequential limit of its elements
is a weak martingale optimal control for problem
\eqref{1.1}--\eqref{1.2}.

\item There is a subsequence $n_k, k \geq 1$ such that 
\[
\widetilde  u_{n_k}^{\,*} \rightharpoonup \widetilde  u \text{ in } L^2(\widetilde\Omega_T;H)
\]
and
\[
\widetilde y_{n_k}
\bigl(\cdot,\widetilde u_{n_k}^{\,*}\bigr)
\to
\widetilde y(\cdot,\widetilde u)
\]
strongly in
$
L^2(\widetilde\Omega_T;H),
$ where 
$
\widetilde y(\cdot,\widetilde u)
$
is an optimal trajectory of problem
\eqref{1.1}--\eqref{1.2}.
\end{enumerate}
\end{theorem}
The next two results concern the optimal control problem on unbounded time intervals (infinite horizon). Denote
\begin{equation}
\mu
:=
-\frac{\Lambda_A}{C_P}
+\Lambda_f
+\lambda
\sup_i \|e_i\|_{L^\infty(D)}^2
L_\sigma^2
+\|a_0\|_{L^\infty(D)} ,
\label{astast}
\end{equation}
where $C_P$ is the Poincar\'e constant for $D$.
Consider 
\eqref{1.3} with 
\begin{equation}\label{2.10}
    \gamma > \mu
\end{equation}
and set
\[
\Omega_\infty :=\Omega\times[0,\infty).
\]

\begin{definition}
An {\it admissible control  for problem
 \eqref{1.1}--\eqref{1.3}}
is any $\mathcal F_t$--adapted random process $u\in L^2(\Omega_\infty;H)$
such that $J(u)<\infty$.
\end{definition}

By Theorem \ref{Th:2.1}, for every admissible control $u$, the
equation \eqref{1.1} has a unique weak solution on
$t\ge0$.

\begin{theorem}\label{Th:2.4}
Assume that conditions \eqref{1.4},
{{\bf \rm(A1)--(A3)}}, and \eqref{2.10} hold.
Then the optimal control problem
\eqref{1.1}--\eqref{1.3}
admits a weak martingale solution

\[
\bigl(
\widetilde u^{\,*}(t),
\widetilde y^{\,*}(t,\widetilde u^{\,*})
\bigr).
\]
\end{theorem}
\begin{definition}
The space $L_\gamma^2(\Omega_\infty;H)$ is the space of $\mathcal F_t$--adapted random processes
$z(t)$ endowed with the norm
\begin{equation*}
\|z\|_\gamma^2
=
\E\int_0^\infty
\|z(t)\|^2 e^{-\gamma t}\,dt .
\end{equation*}
\end{definition}

The next theorem establishes a connection between the optimal
control problems on finite and infinite time intervals.
Fix $T>0$ and consider the optimal control problem on the finite interval
$[0,T]$ with the cost functional

\begin{equation}
J_T
=
\E\int_0^T
\|y(t)\|^2 e^{-\gamma t}\,dt
+
E\int_0^T
\|u(t)\|^2\,dt .
\label{2.11}
\end{equation}
By Theorem \ref{Th:2.2}, this problem admits a weak martingale solution
$
\bigl(
\widetilde u_T^{\,*}(t),
\widetilde y_T^{\,*}(t)
\bigr),
$
and 
$
J_T^{\,*} = J_T(\widetilde u_T^{\,*})
$
is the minimal value of criterion \eqref{2.11}. Using $\widetilde u_T^{\,*}$, we construct an admissible control
for problem \eqref{1.1}--\eqref{1.3} as follows:

\begin{equation}
\widetilde u_{T,\infty}^{\,*}(t)
=
\begin{cases}
\widetilde u_T^{\,*}(t),
& t\in[0,T], \\[1ex]
0,
& t>T .
\end{cases}
\label{2.13}
\end{equation}
This control is admissible for the
infinite-horizon problem.

\begin{theorem}\label{Th:2.5}
Under the assumptions of Theorem \ref{2.4}, the following assertions hold:

\begin{enumerate}
\item
\[
J_T^{*}\to J^{*},
\qquad
T\to\infty.
\]

\item
The family $ \{\widetilde u_{T,\infty}^{\,*}\}$  is bounded in $
L^2(\widetilde\Omega_\infty;H)
$ for all $T \geq 0$. Moreover, for every weakly convergent subsequence
\[
\widetilde u_{T_n,\infty}^{\,*}
\rightharpoonup
\widetilde u,
\qquad
T_n\to\infty,
\]
we have
\[
J\bigl(\widetilde u_{T_n,\infty}^{\,*}\bigr)
\to
J^{*},
\qquad
T_n\to\infty.
\]
Thus, the subsequence $
\{\widetilde u_{T_n,\infty}^{\,*}\}
$ is minimizing for the problem
\eqref{1.1}--\eqref{1.3}.

\item $\widetilde u$ is a weak martingale optimal control for the problem
\eqref{1.1}--\eqref{1.3}.

\item We have
\[
\widetilde y\bigl(\cdot,
\widetilde u_{T_n,\infty}^{\,*}\bigr)
\to
\widetilde y(\cdot,\widetilde u)
\]
weakly in $L_\gamma^2(\widetilde\Omega_\infty;H),$ and $
\widetilde y(t,\widetilde u)
$ is an optimal trajectory of the problem
\eqref{1.1}--\eqref{1.3}.
\end{enumerate}
\end{theorem}

Similar results for equations of type \eqref{1.1}
were obtained in \cite{21,22} in the deterministic setting,
while related results for functional-differential equations were
established in \cite{23}.

\section{Auxiliary Results} \label{Sec:3}

In this section we present several auxiliary results that will be
needed in the sequel. We make use of the following abstract result, which is a
generalization of Skorokhod theorem \cite{24}.

\begin{theorem}[Jakubowski {\cite{25}}] \label{Th:3.1}
Suppose that $(\mathcal{X},\mathcal T)$ is a topological space such that there
exists a countable family
$$\{f_N:\mathcal{X}\to[-1,1], N \geq 1\}$$ of $\mathcal T$--continuous functions separating points of $X$. Furthermore, assume that $\{X_N, N \geq 1\}$ is a sequence of
$\mathcal{X}$--valued random variables and that for every $M\in\mathbb N$
there exists a compact set $K_M\subset X$ such that the tightness conditions holds:
\[
P\{X_N\in K_M\}
>
1-\frac1M,
\qquad
N \geq 1.
\]
Then there exists a subsequence, again denoted by
$(X_N)_{N\in\mathbb N}$, and random variables
\[
\widetilde X,
\widetilde X_N:[0,1]\to \mathcal{X},
\]
where $[0,1]$ is equipped with the Borel $\sigma$--algebra, such
that $\mathcal L(X_N)
=
\mathcal L(\widetilde X_N)$,
and
\[
\widetilde X_N(\omega)\to \widetilde X(\omega),
\qquad
\omega\in[0,1],
\]
where the convergence is understood in the topology
$\mathcal T$.
\end{theorem}

\begin{lemma}\label{lem:3.2}
Suppose the conditions \eqref{1.4} and {{\bf\rm(A1)--(A3)}} hold.
Assume 

\[
 u_n \rightharpoonup u
\qquad\text{weakly in }L^2(\Omega_T;H),
\qquad
n\to\infty,
\]
then, along a subsequence,
\begin{equation} y(t, u_n)
\rightharpoonup
 y(t, u)
\qquad\text{weakly in }
L^2(\Omega_T;H),
\qquad
n\to\infty.
\label{3.1}
\end{equation}
For the ``cut--off'' problem \eqref{1.7} we have a similar result, namely, for every $k \ge 1$,
\begin{equation}
y_k(t, \widetilde u_n)
\rightharpoonup
y_k(t, \widetilde u)
\qquad\text{weakly in }
L^2(\Omega_T;H),
\qquad
n\to\infty.
\label{3.2}
\end{equation}
\end{lemma}

\begin{proof}
Let $y_k(t,u_n), \, y_k(t,u),$ and $y(t,u_n),\, y(t,u)$
be the solutions of the equations \eqref{1.7} and \eqref{1.1},
respectively. Recall that, by Theorem \ref{Th:2.1}, these solutions exist and are unique. We start with establishing  \eqref{3.2}. Since
\[
u_n \rightharpoonup u
\qquad\text{weakly in }L^2(\Omega_T;H),
\]
we have
\begin{equation}
E\int_0^T \|u_n(t)\|^2\,dt \le C,
\label{3.3}
\end{equation}
for some constant $C>0$ independent of $n$. Thus, using the energy equality for \eqref{1.7}, e.g. \cite[Theorem 2.1]{20}, we obtain 
\begin{equation}
\sup_{t\in[0,T]}
\E\|y_k(t,u_n)\|^2
+
\int_0^T
\E\|y_k(t,u_n)\|_V^2\,dt
\le
C_1(y_0,T,C),
\label{3.4}
\end{equation}
uniformly in $k$ and $n$, and
\begin{equation}
\E\int_0^T
\|\sigma(y_k(t,u_n))\|^2\,dt
\le
C_2(y_0,T,C),
\label{3.5}
\end{equation}
uniformly in $k$ and $n$. Therefore, for every fixed $k$, passing to a subsequence if
necessary, we obtain
\[
y_k(\cdot,u_n)
\rightharpoonup
z_k(\cdot)
\qquad\text{weakly in }
L^2(\Omega_T;V),
\qquad
n\to\infty,
\]
and
\[
\sigma(y_k(\cdot,u_n))
\rightharpoonup
\Phi_k(\cdot)
\qquad\text{weakly in }
L^2(\Omega_T;L_2^0),
\qquad
n\to\infty.
\]
In a similar way to \cite{20}, we consider the function
\[
F_k(s)
:=
-\int_0^s f_k(r)\,dr .
\]
Then
\[
F_k'(s)=-f_k(s),
 \text{  and   } 
F_k''(s)=-f_k'(s).
\]
Define
\[
G(y_k(t,u_n))
:=
\int_D
F_k\bigl(y_k(t,u_n)\bigr)\,dx .
\]
Applying It\^o's formula to $G$, we obtain
\begin{align}
& \nonumber G(y_k(t,u_n))-G(y_0)
 =
\int_0^t
\Bigl\langle
A_1y_k(s,u_n),
\,-f_k(y_k(s,u_n))
\Bigr\rangle ds
\\
\nonumber &\quad
-
\int_0^t
\bigl(
a_0y_k(s,u_n),
\,f_k(y_k(s,u_n))
\bigr)\,ds
-
\int_0^t
\bigl(
f_k(y_k(s,u_n)),
\,u_n(s)
\bigr)\,ds
\\
\nonumber &\quad
-
\sum_{i=1}^{\infty}
\lambda_i
\int_0^t
\int_D
f_k(y_k(s,u_n))
\sigma(y_k(s,u_n))
e_i(x)\,dx
\,d\beta_i(s) \\
&\quad
-
\frac12
\sum_{i=1}^{\infty}
\lambda_i^2
\int_0^t
\int_D
f_k'(y_k(s,u_n))
\sigma^2(y_k(s,u_n))
e_i^2(x)\,dx\,ds -\int_0^t
\|f_k(y_k(s,u_n))\|^2\, ds.
\label{3.6ast}
\end{align}
Using \eqref{1.4}, \eqref{2.3}, the Poincar\'e inequality,
and the boundedness of the coefficients $a_{ij}(x)$, we deduce that
\[
-\bigl\langle
A_1y_k,
f_k(y_k)
\bigr\rangle
\le
\widetilde C\,\|y_k\|_V^2,
\]
for some $\widetilde C>0$. Moreover,
\[
\bigl|
(a_0 y_k,f_k(y_k))
\bigr|
\le
\frac{\widetilde C_1}{\varepsilon}
\|y_k\|^2
+
\frac{\varepsilon}{2}
\|f_k(y_k)\|^2,
\]
and
\[
\bigl|
(f_k(y_k),u_n)
\bigr|
\le
\frac{\varepsilon}{2}
\|f_k(y_k)\|^2
+
\frac{1}{2\varepsilon}
\|u_n\|^2
\]
for some positive constants $\widetilde C_1$ and $\varepsilon>0$.
Using the mean value theorem together with
the assumption \eqref{2.3}, we obtain the following estimates
\[
G(y_k)
\ge
-\frac{\lambda_f}{2}\|y_k\|^2,
\]
and
\[
|G(y_0)|
\le
\widetilde C_2
\|y_0\|^2_{L^\infty(\Omega;L^\infty(D))},
\]
where
$ \widetilde C_2 = \widetilde C_2(y_0,f) >0.
$
(see \cite[proof of Theorem~2.1]{20} for details). Next, taking the expectation in \eqref{3.6ast}, for some positive constants we obtain
\begin{align*}
& -\frac{\lambda_f}{2}
E\|y_k(t,u_n)\|^2
\le
C_3
E\|y_0\|^2_{L^\infty(\Omega;L^\infty(D))}
+
C_4
\int_0^t
E\|y_k(s,u_n)\|_V^2\,ds \\
&\quad
+\frac{C_5}{\varepsilon}
\int_0^t
E\|y_k(s,u_n)\|^2\,ds
+
\varepsilon
\int_0^t
E\|f_k(y_k(s,u_n))\|^2\,ds
\nonumber\\
&\quad
+
\frac{T}{2\varepsilon}
\int_0^T
E\|u_n(s)\|^2\,ds
+
\frac12
\sum_{i=1}^{\infty}\lambda_i^2
E
\int_0^t
\int_D
f_k'(y_k(s,u_n(s)))
\sigma^2(y_k(s,u_n(s)))
e_i^2(x)\,dx\,ds
\nonumber \\
&\quad
-E\int_0^t
\|f_k(y_k(s,u_n(s)))\|^2\,ds.
\nonumber
\end{align*}
Combining \eqref{2.4} and \eqref{3.4}, we obtain
\begin{equation}
E\int_0^t
\|f_k(y_k(s,u_n(s)))\|^2\,ds
\le
C_6,
\label{3.6}
\end{equation}
where $C_6$ depends only on $y_0$ and $T$. Thus, passing to a subsequence if necessary, for any $k \geq 1$ we have
\begin{equation}
f_k(y_k(s,u_n(s)))
\rightharpoonup
\psi_k
\qquad
\text{weakly in }
L^2(\Omega_T;H),
\qquad
n\to\infty .
\label{3.7}
\end{equation}
Since both $\sigma$ and $f_k$ satisfy the Lipschitz condition
(and therefore the monotonicity condition) in a similar way to
\cite[Chapter~7, Theorem~7.5]{26} or
\cite[Theorem~4.24]{27}),
we conclude that
\[
f_k(y_k(t,u_n))
\rightharpoonup
f_k(z_k(t)),
\qquad
\sigma(y_k(t,u_n))
\to
\sigma(z_k(t)),
\qquad
n\to\infty .
\]
Next, similarly to \cite{20}, using Fubini's theorem and the
definition of a weak solution, for every
$\varphi\in L^\infty(\Omega_T)$ and 
$v\in V,
$
we may pass to the limit and obtain
\begin{align*}
& \E\int_0^T
\langle z_k(t),\varphi(t)v \rangle \,dt
=
\lim_{n\to\infty}
\E
\Biggl[
\int_0^T
\langle y_0,\varphi(t)v \rangle \,dt
\\
&\quad
+
\int_0^T
\int_0^t
\langle
A_1y_k(s,u_n(s)),
\varphi(t)v
\rangle
\,ds\,dt
+
\int_0^T
\int_0^t
(a_0y_k(s,u_n(s)),
\varphi(t)v)
\,ds\,dt
\\
&\quad
+
\int_0^T
\int_0^t
(u_n(s),\varphi(t)v)
\,ds\,dt
+
\int_0^T
\int_0^t
(f_k(y_k(s,u_n(s))),
\varphi(t)v)
\,ds\,dt
\\
&\quad
+
\int_0^T
\left(
\int_0^t
\sigma(y_k(s,u_n(s)))\,dW(s),
\varphi(t)v
\right)
dt
\Biggr].
\end{align*}
Hence $z_k(t)$ satisfies equation \eqref{1.7} with control
$u(t)$, and by uniqueness of solutions,
we obtain $z_k(t)=y_k(t,u)$. Thus, the statement \eqref{3.2} follows.

We now prove \eqref{3.1}. Let $n \ge 1$ be fixed. From \eqref{3.4}, \eqref{3.5}, and \eqref{3.6}, passing to a subsequence if necessary, we obtain
\begin{equation}
\begin{aligned}
y_k(\cdot,u_n)
&\rightharpoonup
z_k(\cdot)
&&\text{weakly in }
L^2(\Omega_T;V),\\
\sigma(y_k(\cdot,u_n))
&\rightharpoonup
\Phi_n(\cdot)
&&\text{weakly in }
L^2(\Omega_T;L_2^0),\\
f_k(y_k(\cdot,u_n))
&\rightharpoonup
\psi_n(\cdot)
&&\text{weakly in }
L^2(\Omega_T;H),
\end{aligned}
\qquad k\to\infty.
\label{3.8}
\end{equation}

Next, similarly to \cite{20}, using the stochastic Fubini theorem,
the factorization formula (see \cite{19}), and compactness of the
operator 
\[
(G_\alpha\varphi)(t)
=
\int_0^t
(t-s)^{\alpha-1}
S(t-s)\varphi(s)\,ds,
\qquad
\alpha\in\left(\frac12,1\right],
\]
as a map from
$L^2(0,T;H)$ to $
C([0,T];H),
$
we conclude that the sequence $\{\mathcal L(y_k(\cdot,u_n))\}$
is relatively compact with respect to weak convergence of measures on
\(C([0,T];H)\). Therefore, passing to a subsequence if necessary,
$\mathcal L(y_k(\cdot,u_n))
$
weakly converges to $\mu$. Next, we show that the family $\{u_n\}$ is tight. By Chebyshev's inequality and
\eqref{3.3},
\[
P\!\left(
\|u_n\|_{L^2((0,T);H)}
>R
\right)
\le
\frac1{R^2}
E\int_0^T
\|u_n(t)\|^2\,dt
\to 0,
\qquad
R\to\infty,
\]
uniformly in \(n\). The closed ball
\[
\left\{
u\in L^2((0,T);H):
\|u\|_{L^2((0,T);H)}\le R
\right\}
\]
is weakly compact in \(L^2((0,T);H)\) as well. Thus, by Theorem \ref{Th:3.1}, there exist a probability space
$(\widetilde\Omega,\widetilde{\mathcal F},\widetilde P)$,
a filtration $(\widetilde{\mathcal F}_t)$, an
$(\widetilde{\mathcal F}_t)$-adapted Wiener process
$\widetilde W(t)$, and random variables
$\widetilde u_n$, $\widetilde y_k$, $\widetilde z_n$, and
$\widetilde y_0$ such that

\[
\mathcal L(y_k)=\mathcal L(\widetilde y_k),\qquad
\mathcal L(\widetilde z_n)=\mu_n,\qquad
\mathcal L(u_n)=\mathcal L(\widetilde u_n),\qquad
\mathcal L(y_0)=\mathcal L(\widetilde y_0),
\]
with
$\widetilde y_k\to\widetilde z_n$ in $C([0,T];H)$ as
$k\to\infty$, and
$\widetilde u_n\rightharpoonup\widetilde u$ weakly in
$L^2((0,T);H)$ as $n\to\infty$, $\widetilde P$-a.s.
(passing to a subsequence if necessary). Since
$\mathcal L(y_k,u_n,W,y_0)
=\mathcal L(\widetilde y_k,\widetilde u_n,\widetilde W,\widetilde y_0)$,
the process $\widetilde y_k$ satisfies

\begin{equation}
\widetilde y_k
=
\widetilde y_0
+
\int_0^t
\bigl(
A_1\widetilde y_k
+
f(\widetilde y_k)
+
\widetilde u_n
\bigr)\,ds
+
\int_0^t
\sigma(\widetilde y_k)\,d\widetilde W(s)
\label{3.9}
\end{equation}
in $V$. For fixed $n \geq 1$, since $\widetilde{y_k}(\cdot, \widetilde{u_n})$ is bounded in $H$, for $k$ large enough $f_k(\widetilde{y_k}(\cdot, \widetilde{u_n})) \equiv f(\widetilde{y_k}(\cdot, \widetilde{u_n}))$, thus using Lemma 1.3 \cite{28} we have,
\begin{equation}\label{3.10}
  f_k(\widetilde{y_k}(\cdot, \widetilde{u_n})) \rightharpoonup f(\widetilde{z_n}) \ \text{ weakly in } L^2(\widetilde \Omega_T, H), \ k \to \infty, 
\end{equation}
and 
\begin{equation}\label{3.11}
  \sigma_k(\widetilde{y_k}(\cdot, \widetilde{u_n})) \rightharpoonup \sigma(\widetilde{z_n}) \ \text{ weakly in } L^2(\widetilde \Omega_T, L_2^0), \ k \to \infty.    
\end{equation}
Next, similarly to \cite{20}, using Fubini's theorem and the definition
of a weak solution, for every
$\varphi\in L^\infty(\widetilde\Omega_T)$ and $v\in V$, passing to the limit in the identity  

\begin{align*}
& \widetilde E\int_0^T
(\widetilde z_n(t),\varphi(t)v)\,dt
= \lim_{k\to\infty}
\widetilde E
\int_0^T
(\widetilde y_k,\varphi(t)v)\,dt = 
\lim_{k\to\infty}
\widetilde E
\Biggl[
\int_0^T
(\widetilde y_0,\varphi(t)v)\,dt
\\
&\quad
+
\int_0^T\!\!\int_0^t
\langle
A_1\widetilde y_k(s),
\varphi(t)v
\rangle
\,ds\,dt
+
\int_0^T\!\!\int_0^t
(a_0\widetilde y_k(s),
\varphi(t)v)
\,ds\,dt
\\
&\quad
+
\int_0^T\!\!\int_0^t
(f(\widetilde y_k(s)),
\varphi(t)v)
\,ds\,dt
+
\int_0^T\!\!\int_0^t
(\widetilde u_n(s),
\varphi(t)v)
\,ds\,dt
\\
&\quad
+
\int_0^T
\left(
\int_0^t
\sigma(\widetilde y_k(s))
\,d\widetilde W(s),
\varphi(t)v
\right) \, dt
\Biggr],
\end{align*}
we obtain that $\widetilde{z}_n$ is a martingale weak solution of
equation \eqref{1.1} corresponding to the control $\widetilde{u}_n$,
that is, $\widetilde{z}_n=y(t,\widetilde{u}_n)$. By \eqref{3.6} and
Fatou's lemma, we have
\begin{equation}
\widetilde{\mathbb{E}}
\int_0^T
\left\|
f\bigl(\widetilde{y}(t,\widetilde{u}_n(t))\bigr)
\right\|^2\,dt
\leq C_6.
\label{3.11}
\end{equation}
Moreover, analogously to the preceding arguments, we obtain
\[
\sup_{t\in[0,T]}
\widetilde{\mathbb{E}}
\left\|
\widetilde{y}(t,\widetilde{u}_n)
\right\|^2
+
\int_0^T
\widetilde{\mathbb{E}}
\left\|
\widetilde{y}(t,\widetilde{u}_n(t))
\right\|_V^2\,dt
\leq
C_7(\widetilde{y}_0,T,C),
\]
and
\[
\widetilde{\mathbb{E}}
\int_0^T
\left\|
\sigma\bigl(\widetilde{y}(t,\widetilde{u}_n)\bigr)
\right\|^2\,dt
\leq
C_8(\widetilde{y}_0,T,C_1).
\]
Therefore,
\[
\widetilde{z}_n
\rightharpoonup
\widetilde{z}
\quad\text{weakly in }
L^2(\widetilde{\Omega}_T;V),
\]
\[
\sigma(\widetilde{z}_n)
\rightharpoonup
\widetilde{\Phi}
\quad\text{weakly in }
L^2(\widetilde{\Omega}_T;L_2^0),
\]
\[
f(\widetilde{z}_{u_n})
\rightharpoonup
\widetilde{\Psi}
\quad\text{weakly in }
L^2(\widetilde{\Omega}_T;H),
\]
and
\[
\widetilde{u}_n
\rightharpoonup
\widetilde{u}
\quad\text{weakly in }
L^2(\widetilde{\Omega}_T;H)
\]
as $n\to\infty$.
Next, analogously to the arguments above, we establish the weak
compactness of $\mathcal{L}(\widetilde{z}_n)$ in $C([0,T];H)$
and the weak compactness of $\{\mathcal{L}(\widetilde{u}_n)\}$.
Hence, again by Theorem \ref{Th:3.1}, there exist a probability space
\[
\bigl(
\wwtilde{\Omega},
\wwtilde{\mathcal{F}},
\wwtilde{P}
\bigr),
\]
a Wiener process $\wwtilde{W}(t)$, and random variables
\[
\wwtilde{u}_n,\quad
\wwtilde{z}_n,\quad
\wwtilde{u},\quad
\wwtilde{z},\quad
\wwtilde{y}_0
\]
such that
\[
\mathcal{L}
\bigl(
\wwtilde{u}_n,
\wwtilde{z}_n, \wwtilde{W}
\wwtilde{y}_0
\bigr)
=
\mathcal{L}
\bigl(
\widetilde{u}_n,
\widetilde{z}_n, \widetilde{W},
\widetilde{y}_0
\bigr),
\]
\[
\mathcal{L}(\wwtilde{z})
=
\mathcal{L}(\widetilde{z}),
\qquad
\mathcal{L}(\wwtilde{u})
=
\mathcal{L}(\widetilde{u}),
\]
and, as $n\to\infty$,
\[
\wwtilde{z}_n
\longrightarrow
\wwtilde{z}
\quad\text{in }C([0,T];H),
\qquad
\wwtilde{P}\text{-a.s.},
\]
as well as
\[
\wwtilde{u}_n
\rightharpoonup
\wwtilde{u}
\quad\text{weakly in }L^2((0,T);H),
\qquad
\wwtilde{P}\text{-a.s.}
\]

From Lemma 1.3 \cite{28}, we also obtain
\[
f(\wwtilde{z}_n)
\longrightarrow
f(\wwtilde{z}),
\qquad n\to\infty,
\qquad
\wwtilde{P}\text{-a.s.}
\]

Then, passing to the limit in the equality 
\[
\begin{aligned}
& \bigl(\wwtilde{z}(t),\varphi(t)v\bigr)
=
\lim_{n\to\infty}
\wwtilde{\mathbb{E}}
\Bigg[
\int_0^T
\left(
\wwtilde{y}_0,\varphi(t)v
\right) dt +
\int_0^T\int_0^t
\left\langle
A_1\wwtilde{z}_n(s),\varphi(t)v
\right\rangle ds\,dt
\\
&\quad+
\int_0^T\int_0^t
\bigl(
a_0\wwtilde{z}_n(s),\varphi(t)v
\bigr)\,ds\,dt +
\int_0^T\int_0^t
\bigl(
f(\wwtilde{z}_n(s)),\varphi(t)v
\bigr)\,ds\,dt
\\
&\quad+
\int_0^T\int_0^t
\bigl(
\wwtilde{u}_n(s),\varphi(t)v
\bigr)\,ds\,dt +
\int_0^T\int_0^t
\bigl(
\sigma(\wwtilde{z}_n(s))\,d\wwtilde{W}(s),
\varphi(t)v
\bigr)\,dt
\Bigg],
\end{aligned}
\]
we obtain that $\wwtilde{z}(t)$ is the solution of equation
\eqref{1.1} corresponding to $\wwtilde{u}$, that is,
\[
\wwtilde{z}(t)
=
\wwtilde{y}(t,\wwtilde{u}).
\]
The proof of Lemma \ref{lem:3.2} is complete.
\end{proof}

\begin{lemma}\label{lem:3.3}
Under condition \eqref{1.4} and assumptions \textnormal{(A1)--(A3)},
for every weak admissible control $\widetilde{u}$, we have
\begin{equation}
\widetilde{y}_k(\cdot,\widetilde{u})
\longrightarrow
\widetilde{y}(\cdot,\widetilde{u})
\quad\text{in }
L^2(\widetilde{\Omega}_T;H),
\qquad k\to\infty.
\label{3.11}
\end{equation}
\end{lemma}
\begin{proof}
Let $\widetilde{u}\in L^2(\widetilde{\Omega}_T;H)$. Analogously to
the proof of the preceding lemma, there exists a probability space
$(\overline{\Omega},\overline{\mathcal{F}},\overline{P})$ and
corresponding random variables such that
\begin{equation}\label{3.12}
\tilde{y}_k(\cdot,\tilde{u})
\rightharpoonup
\tilde{y}(\cdot,\tilde{u})
\quad\text{ weakly in }
L^2(\overline{\Omega}_T;H),
\qquad k\to\infty
\end{equation}
and
\begin{equation}
\widetilde{y}_k(\cdot,\widetilde{u})
\longrightarrow
\widetilde{y}(\cdot,\widetilde{u})
\quad\text{in }C([0,T];H),
\qquad k\to\infty,
\qquad \widetilde{P}\text{-a.s.}
\label{3.13}
\end{equation}
For any $M>0$, we define
\begin{equation}
\widetilde{u}_M(t,x,\widetilde{\omega})
=
\begin{cases}
\widetilde{u}(t,x,\widetilde{\omega}),
& \left|\widetilde{u}(t,x,\widetilde{\omega})\right|\leq M,\\
M,
& \text{otherwise}.
\end{cases}
\label{3.14}
\end{equation}
Using the dominated convergence theorem, we have
\begin{equation}
\widetilde{\mathbb{E}}
\int_0^T
\left\|
\widetilde{u}_M(t)-\widetilde{u}(t)
\right\|^2\,dt
\longrightarrow 0,
\qquad M\to\infty.
\label{3.15}
\end{equation}

Denote $\widetilde{y}_k(t,\widetilde{u}_M)$ and
$\widetilde{y}(t,\widetilde{u}_M)$ to be the weak martingale solutions of
\eqref{1.7} and \eqref{1.1}, respectively, corresponding to the
admissible control $\widetilde{u}_M$. By definition,
\[
\left|
\widetilde{u}_M(t,x,\widetilde{\omega})
\right|
\leq M.
\]

Now we show that
\begin{equation} \label{3.16}
\widetilde{\mathbb{E}}
\int_0^T
\|
\widetilde{y}_k(t,\widetilde{u}_M)\|^2 \to \widetilde{\mathbb{E}}
\int_0^T
\|
\widetilde{y}(t,\widetilde{u}_M)\|^2,
\qquad k\to\infty.
\end{equation}
We will be using the following estimates:
\[
\left\langle
A_1\widetilde{y}_k,\widetilde{y}_k
\right\rangle
\leq
-\Lambda_A
\left\|
\widetilde{y}_k
\right\|_V^2,
\tag{*1}
\]
\[
\begin{aligned}
\left\|
\sigma(\widetilde{y}_k)
\right\|_{L_2^0}^2
&=
\sum_i \lambda_i^2
\int_D
\sigma_i^2(\widetilde{y}_k)
e_i(x)^2\,dx \leq
\lambda \sup_{i \geq 1}
\sup_D |e_i|^2
L_\sigma^2
\left(
1+\left\|\widetilde{y}_k\right\|^2
\right),
\end{aligned}
\tag{*2}
\]
\[
\bigl(
f_k(\widetilde{y}_k),\widetilde{y}_k
\bigr)
=
\int_D
\widetilde{y}_k
f_k(\widetilde{y}_k)\,dx
\leq
\Lambda_f
\left\|
\widetilde{y}_k
\right\|^2,
\tag{*3}
\]
and
\[
\bigl(
a_0\widetilde{y}_k,\widetilde{y}_k
\bigr)
\leq
C_8
\left\|
\widetilde{y}_k
\right\|^2,
\qquad
\bigl(
\widetilde{u}_M,\widetilde{y}_k
\bigr)
\leq
\frac{1}{2}
\left(
M^2+\left\|\widetilde{y}_k\right\|^2
\right).
\tag{*4}
\]
To simplify the notation, let us denote $\widetilde y_k:= y_k$.
For $p\geq 4$, we apply Itô's formula to
$\|y_k(t,u_M)\|^p$. We have
\begin{align}
\|y_k(t)\|^p
={}&
\|y_0\|^p
+
p\sum_{r=1}^{\infty}\lambda_r
\int_0^t
\|y_k(s)\|^{p-2}
\bigl(
\sigma(y_k(s)),
y_k(s)
\bigr)\,d\beta_r(s)
\nonumber\\
&+
p\int_0^t
\|y_k(s)\|^{p-2}
\Bigl[
\left\langle
A_1 y_k(s),
y_k(s)
\right\rangle
+
\bigl(
a_0 y_k(s),
y_k(s)
\bigr)
\nonumber\\
&\hspace{4cm}
+
\bigl(
\widetilde u_M(s),
y_k(s)
\bigr)
+
\bigl(
f_k(y_k(s)),
y_k(s)
\bigr)
\Bigr]\,ds
\nonumber\\
&+
\frac{p}{2}
\int_0^t
\|y_k(s)\|^{p-2}
\left\|
\sigma(y_k(s))
\right\|_{L_2^0}^2\,ds
\nonumber\\
&+
\frac{p(p-2)}{2}
\sum_{r=1}^{\infty}\lambda_r^2
\int_0^t
\|y_k(s)\|^{p-4}
\bigl(
\sigma(y_k(s)),
y_k(s)
\bigr)^2\,ds,
\qquad P\text{-a.s.}
\label{3.22}
\end{align}

Next, for any $R>0$, we introduce the stopping times
\[
\tau_R
:=
\inf
\left\{
t\in[0,T]:
\|y_k(t)\|>R
\right\}
\wedge T.
\]
Clearly 
\begin{equation}
\tau_R \to T
\quad\text{as }R\to\infty,
\qquad
\widetilde{P}\text{-a.s.}
\label{3.23}
\end{equation}
The Burkholder--Davis--Gundy inequality (see, e.g.
\cite[Theorem 3.28]{29}), implies

\begin{align}\label{3.24}
&\widetilde{\mathbb{E}}
\sup_{s\in[0,t\wedge\tau_R]}
\left|
\int_0^s
\|y_k(r)\|^{p-2}
\bigl(
\sigma(y_k(r)),
y_k(r)
\bigr)\,d\beta_r(r)
\right|
\nonumber\\
& \qquad \leq
3\widetilde{\mathbb{E}}
\left(
\int_0^{t\wedge\tau_R}
\|y_k(s)\|^{2p-4}
\bigl(
\sigma(y_k(s)),
y_k(s)
\bigr)^2\,ds
\right)^{1/2}.
\end{align}
Using the assumption \textnormal{(A2)}, we obtain
\begin{align}
\bigl(
\sigma(y_k(s)),
y_k(s)
\bigr)^2
&\leq
\int_D
\left|
\sigma(y_k(s,x))
\right|^2\,dx
\int_D
y_k^2(s,x)\,dx
\nonumber\\
&\leq
L_\sigma^2
\int_D
\left(
1+
|y_k(s,x)|^2
\right)\,dx\,
\|y_k(s)\|^2
\nonumber\\
&\leq
C_8\|y_k(s)\|^2
+
C_9\|y_k(s)\|^4,
\label{3.25}
\end{align}
where the constants $C_8>0$ and $C_9>0$ depend only on
$L_\sigma$ and $\operatorname{meas}(D)$. Using \eqref{3.25}, the expression in \eqref{3.24} may be estimated as follows:
\begin{align}
& \widetilde{\mathbb{E}}
\left(
\int_0^{t\wedge\tau_R}
\|y_k(s)\|^{2p-4}
\bigl(
\sigma(y_k(s)),
y_k(s)
\bigr)^2\,ds
\right)^{1/2} \\
\nonumber
& \leq 
C_{10}\widetilde{\mathbb{E}}
\left(
\int_0^{\tau_R\wedge t}
\|y_k(s)\|^{2p-2}\,ds
+
\int_0^{\tau_R\wedge t}
\|y_k(s)\|^{2p}\,ds
\right)^{1/2}
\nonumber\\
& \leq
C_{10}
\left[
\widetilde{\mathbb{E}}
\left(
\int_0^{\tau_R\wedge t}
\|y_k(s)\|^{2p-2}\,ds
\right)^{1/2}
+
\widetilde{\mathbb{E}}
\left(
\int_0^{\tau_R\wedge t}
\|y_k(s)\|^{2p}\,ds
\right)^{1/2}
\right] =: I_1+I_2.
\label{3.26}
\end{align}

For $I_1$ we may conclude that for some
constant $C_{11}>0$ we have
\begin{equation}
I_1
\leq
C_{11}
\left(
\mathbb{E}
\int_0^{\tau_R\wedge t}
\|y_k(s)\|^p\,ds
+1
\right).
\label{3.27}
\end{equation}
To estimate $I_2$, by Young's inequality,
\begin{align}
I_2
&\leq
C_{10}\widetilde{\mathbb{E}}
\left(
\sup_{s\in[0,\tau_R\wedge t]}
\|y_k(s)\|^p
\int_0^{\tau_R\wedge t}
\|y_k(s)\|^p\,ds
\right)^{1/2}
\nonumber\\
&\leq
C_{12}\varepsilon
\widetilde{\mathbb{E}}
\sup_{s\in[0,\tau_R\wedge t]}
\|y_k(s)\|^p
+
\frac{C_{13}}{\varepsilon}
\widetilde{\mathbb{E}}
\int_0^{\tau_R\wedge t}
\|y_k(s)\|^p\,ds,
\label{3.28}
\end{align}
where $\varepsilon>0$ can be chosen arbitrarily small.
Next, we estimate \eqref{3.22}:
\begin{align}
\int_0^t
\|{y}_k(s)\|^{p-4}
\bigl(
\sigma({y}_k(s)),{y}_k(s)
\bigr)^2\,ds
&\leq
C_{15}
\left(
\int_0^t
\|{y}_k(s)\|^{p-2}\,ds
+
\int_0^t
\|{y}_k(s)\|^p\,ds
\right)
\nonumber\\
&\leq
C_{16}
\left(
\int_0^t
\|{y}_k(s)\|^p\,ds
+1
\right).
\label{3.29}
\end{align}

The remaining terms in \eqref{3.22} are estimated analogously,
using estimates \textnormal{(*1)--(*4)}. Thus, from
\eqref{3.28}, \eqref{3.29}, and \eqref{3.22}, we obtain
\[
\begin{aligned}
\widetilde{\mathbb{E}}
\sup_{s\in[0,t\wedge\tau_R]}
\|{y}_k(s,\widetilde{u}_M)\|^p
\leq
C_{17}
\Biggl(
\widetilde{\mathbb{E}}\|{y}_0\|^p
+
\int_0^t
\widetilde{\mathbb{E}}
\sup_{\nu\in[0,s\wedge\tau_R]}
\|{y}_k(\nu,\widetilde{u}_M)\|^p\,ds
\Biggr),
\end{aligned}
\]
where $C_{17}$ is independent of $M$. Hence, Gronwall's inequality implies
\[
\widetilde{\mathbb{E}}
\sup_{s\in[0,t\wedge\tau_R]}
\|{y}_k(s,\widetilde{u}_M)\|^p
\leq
C_{18}\widetilde{\mathbb{E}}\|{y}_0\|^p.
\]
Passing to the limit $R \to \infty$ using Fatou's lemma, we get 
\begin{equation}
\widetilde{\mathbb{E}}
\sup_{s\in[0,T]}
\|\widetilde{y}_k(s,\widetilde{u}_M)\|^p
\leq
C_{18}\widetilde{\mathbb{E}}\|\widetilde{y}_0\|^p.
\label{3.30}
\end{equation}
Thus, using Vitali's convergence theorem, we get \eqref{3.16}, which, in conjunction with \eqref{3.12}, for every $M>0$ yields 
\begin{equation}
{y}_k(\cdot,\widetilde{u}_M)
\longrightarrow
{y}(\cdot,\widetilde{u}_M)
\quad\text{in }L^2(\widetilde{\Omega}_T;H),
\qquad k\to\infty.
\label{3.31}
\end{equation}
For any $k$, by Itô's formula, Gronwall's inequality and \eqref{3.15}, we obtain
\begin{equation}
\sup_k \sup_{t\in[0,T]}
\widetilde{\mathbb{E}} \left\|
{y}_k(t,\widetilde{u})
-
{y}_k(t,\widetilde{u}_M)
\right\|^2
\leq
C_{19}
\int_0^T
\widetilde{\mathbb{E}}
\left\|
\widetilde{u}(s)-\widetilde{u}_M(s)
\right\|^2\,ds
\to 0,
\qquad M\to\infty.
\label{3.32}
\end{equation}
In a similar way, we obtain
\begin{equation}
\widetilde{\mathbb{E}}
\sup_{t\in[0,T]}
\left\|
{y}(t,\widetilde{u})
-
{y}(t,\widetilde{u}_M)
\right\|^2
\to 0,
\qquad M\to\infty.
\label{3.33}
\end{equation}
This way
\begin{align}
& \widetilde{\mathbb{E}}
\int_0^T
\left\|
{y}_k(t,\widetilde{u})
-
{y}(t,\widetilde{u})
\right\|^2\,dt
\leq
3\widetilde{\mathbb{E}}
\int_0^T
\left\|
{y}_k(t,\widetilde{u})
-
{y}_k(t,\widetilde{u}_M)
\right\|^2\,dt
\nonumber\\
&+
3\widetilde{\mathbb{E}}
\int_0^T
\left\|
{y}_k(t,\widetilde{u}_M)
-
{y}(t,\widetilde{u}_M)
\right\|^2\,dt
\nonumber +
3\widetilde{\mathbb{E}}
\int_0^T
\left\|
{y}(t,\widetilde{u}_M)
-
{y}(t,\widetilde{u})
\right\|^2\,dt
=: I_1+I_2+I_3.
\label{3.34}
\end{align}
In view of \eqref{3.32} and \eqref{3.33}, for any $\varepsilon>0$, choose $M=M(\varepsilon)$ sufficiently
large, so that uniformly in $k \geq 1$
\[
I_1<\frac{\varepsilon}{3},
\qquad \text{ and }
I_3<\frac{\varepsilon}{3}.
\]
For such $M$, using \eqref{3.31}, we may now choose
$k$ sufficiently large so that
\[
I_2<\frac{\varepsilon}{3}.
\]
Thus, the proof of Lemma \ref{lem:3.3} is now complete.
\end{proof}
\begin{lemma}\label{lem:3.4}
Let the condition \eqref{1.4} and assumptions \textnormal{(A1)--(A3)}
hold. If the admissible controls satisfy
\[
u_k \rightharpoonup u
\quad\text{weakly in }L^2(\Omega_T;H),
\qquad k\to\infty,
\]
then, along a subsequence,
\[
\widetilde{y}_k(\cdot,\widetilde{u}_k)
\rightharpoonup
\widetilde{y}(\cdot,\widetilde{u})
\quad\text{weakly in }
L^2(\widetilde{\Omega}_T;H),
\qquad k\to\infty.
\]
\end{lemma}

\begin{proof}
The proof of this lemma follows the same argument as the proof of
Lemma \ref{lem:3.2}.
\end{proof}
\begin{lemma}\label{lem:3.5}
For any solution of \eqref{1.1}, we have the following estimate:
\begin{equation}
\mathbb{E}\|y(t)\|^2
\leq
C_{20}
\left(
\mathbb{E}\|y_0\|_{L^{\infty}}^2 
+
\int_0^t
\mathbb{E}\|u(s)\|^2\,ds
+ \lambda \sup_{i}\|e_i\|_{L^\infty}^2 L_{\sigma}^2 t \right) e^{mt}.
\label{3.35}
\end{equation}
for any $t>0$, where $C_{20}>0$ is independent of $t$, and $m$
is defined in \eqref{astast}.
\end{lemma}

\begin{proof}
From the energy equality and estimates \textnormal{(*1)--(*4)}, we
obtain
\[
\begin{aligned}
\mathbb{E}\|y(t)\|^2
\leq{}&
\mathbb{E}\|y_0\|^2
+
2\mathbb{E}\int_0^t
\Bigl(
-\Lambda_A\|y(s)\|_V^2
+\Lambda_f\|y(s)\|^2
\\
&\qquad\qquad
+\lambda\sup_i\|e_i\|_{L^\infty}^2
L_\sigma^2\bigl(1+\|y(s)\|^2\bigr)
+\|a_0\|_{L^\infty}\|y(s)\|^2
\\
&\qquad\qquad
+\frac{1}{\ve}\|u(s)\|^2
+\varepsilon\|y(s)\|^2
\Bigr)\,ds
\end{aligned}
\]
for every $\varepsilon>0$.  Poincaré's inequality
gives
\[
-\Lambda_A\|y\|_V^2
\leq
-\frac{\Lambda_A}{C_P}\|y\|^2.
\]
Thus, the estimate \eqref{3.35} follows by choosing a sufficiently small
$\varepsilon>0$ and applying Gronwall's inequality.
\end{proof}

\section{Proofs of the Main Results}\label{Sec:4}

In this section, we prove Theorems \ref{Th:2.2}--\ref{Th:2.5}.

\begin{proof}[Proof of Theorem \ref{Th:2.2}] We start with noting that the set of admissible controls is non-empty, as $u \equiv 0 \in U_{\mathrm{ad}}$. Since $J(u)\geq 0$, there exists
$J^*\geq 0$ such that
\begin{equation}
J^*
=
\inf_{u\in U_{\mathrm{ad}}}J(u).
\label{4.1}
\end{equation}

Hence, there exists a minimizing sequence of admissible controls
$\{u_n\}$ such that
\begin{equation}
J(u_n)\longrightarrow J^*,
\qquad n\to\infty.
\label{4.2}
\end{equation}
Thus, by \eqref{4.1}, for all sufficiently large $n$, we have
\[
\begin{aligned}
J^*
\leq J(u_n)
=
\mathbb{E}\int_0^T
\|y(t,u_n)\|^2\,dt
+
\mathbb{E}\int_0^T
\|u_n(t)\|^2\,dt
\leq J^*+1.
\end{aligned}
\]
This implies that $\{u_n\}$ is weakly compact in $L^2(\Omega_T;H)$.
Passing to a subsequence, if necessary, we obtain
\begin{equation}
u_n\rightharpoonup u
\quad\text{weakly in }L^2(\Omega_T;H),
\qquad n\to\infty.
\label{4.3}
\end{equation}
Thus, by Lemma \ref{lem:3.2}, there exist a probability space
$(\widetilde{\Omega},\widetilde{\mathcal{F}},\widetilde{P})$ and corresponding 
random variables, such that
\[
\widetilde{y}(t,\widetilde{u}_n)
\rightharpoonup
\widetilde{y}(t,\widetilde{u})
\quad\text{weakly in }
L^2(\widetilde{\Omega}_T;H),
\qquad n\to\infty.
\]
Therefore, using the lower semicontinuity of the cost functional $J$,
we obtain
\[
\begin{aligned}
J^*
&=
\lim_{n\to\infty}J(\widetilde{u}_n)
\geq
\liminf_{n\to\infty}
\widetilde{\mathbb{E}}
\int_0^T
\left\|
\widetilde{y}(t,\widetilde{u}_n)
\right\|^2\,dt
+
\liminf_{n\to\infty}
\widetilde{\mathbb{E}}
\int_0^T
\left\|
\widetilde{u}_n(t)
\right\|^2\,dt
\\
&\geq
\widetilde{\mathbb{E}}
\int_0^T
\left\|
\widetilde{y}(t,\widetilde{u})
\right\|^2\,dt
+
\widetilde{\mathbb{E}}
\int_0^T
\left\|
\widetilde{u}(t)
\right\|^2\,dt.
\end{aligned}
\]

Consequently, the pair
$\bigl(\widetilde{u},\widetilde{y}(t,\widetilde{u})\bigr)$ is a weak
martingale solution of the optimal control problem
\eqref{1.1}--\eqref{1.2}. Here we used the fact that the
distributions of $u_n$ and $\widetilde{u}_n$, as well as those of
$y(t,u_n)$ and $\widetilde{y}(t,\widetilde{u}_n)$, coincide.
The existence of an optimal pair for problem \eqref{1.7} is proved
analogously.
\end{proof}

\begin{proof}[Proof of Theorem \ref{Th:2.3}]
By Theorem \ref{Th:2.1}, the problems \eqref{1.1}--\eqref{1.2} and \eqref{1.7}
have weak martingale solutions
\[
\bigl(
\widetilde{u}^*,
\widetilde{y}^*(t,\widetilde{u}^*)
\bigr)
\quad\text{and}\quad
\bigl(
\widetilde{u}_k^*,
\widetilde{y}_k^*(t,\widetilde{u}_k^*)
\bigr),
\]
respectively. Once again, to simplify the notation, let us drop the superscripts $\sim$.  For every $k$, we have
\[
J_k^*
=
J_k({u}_k^*)
\leq
J_k(0).
\]
Hence,
\begin{equation}
{\mathbb{E}}
\int_0^T
\left\|
{y}_k^*(t,{u}_k^*)
\right\|^2\,dt
+
{\mathbb{E}}
\int_0^T
\left\|
{u}_k^*(t)
\right\|^2\,dt
\leq
{\mathbb{E}}
\int_0^T
\left\|
{y}_k(t,0)
\right\|^2\,dt.
\label{4.4}
\end{equation}
By Lemma \ref{lem:3.3}, we have
\begin{equation}
{y}_k(\cdot,0)
\longrightarrow
{y}(\cdot,0)
\quad\text{in }
L^2({\Omega}_T;H),
\qquad k\to\infty.
\label{4.5}
\end{equation}
Therefore, it follows from \eqref{4.4} that
$\{u_k^*\}$ is weakly compact in
$L^2(\Omega_T;H)$. Hence, choosing a weakly convergent
subsequence, we have
\[
u_k^*
\rightharpoonup
u_0
\quad\text{weakly in }
L^2(\Omega_T;H),
\qquad k\to\infty.
\]
Thus, by Lemma \ref{lem:3.4}, we have
\begin{equation}
y_k^*(\cdot,u_k^*)
\rightharpoonup
y(\cdot,u_0)
\quad\text{weakly in }
L^2(\Omega_T;H),
\qquad k\to\infty.
\label{4.6}
\end{equation}
Using  weak lower semicontinuity, we obtain
\begin{equation}
\liminf_{k\to\infty}
\mathbb{E}
\int_0^T
\|u_k^*(t)\|^2\,dt
\geq
\mathbb{E}
\int_0^T
\|u_0(t)\|^2\,dt.
\label{4.7}
\end{equation}
On the other hand,
\[
J_k(u_k^*)
\leq
J_k(u_0).
\]
Therefore,
\[
\begin{aligned}
&\liminf_{k\to\infty}
\left[
\mathbb{E}
\int_0^T
\left\|
y_k^*(t,u_k^*)
\right\|^2\,dt
+
\mathbb{E}
\int_0^T
\|u_k^*(t)\|^2\,dt
\right]
\\
&\qquad\leq
\liminf_{k\to\infty}
\left[
\mathbb{E}
\int_0^T
\left\|
y_k(t,u_0)
\right\|^2\,dt
+
\mathbb{E}
\int_0^T
\|u_0(t)\|^2\,dt
\right].
\end{aligned}
\]
Using weak lower semicontinuity once again, in view of \eqref{4.5} and \eqref{4.6}, we obtain
\[
\begin{aligned}
&\mathbb{E}
\int_0^T
\left\|
y(t,u_0)
\right\|^2\,dt
+
\liminf_{k\to\infty}
\mathbb{E}
\int_0^T
\|u_k^*(t)\|^2\,dt
\\
&\qquad\leq
\mathbb{E}
\int_0^T
\left\|
y(t,u_0)
\right\|^2\,dt
+
\mathbb{E}
\int_0^T
\|u_0(t)\|^2\,dt.
\end{aligned}
\]
Consequently,
\begin{equation}
\liminf_{k\to\infty}
\mathbb{E}
\int_0^T
\|u_k^*(t)\|^2\,dt
\leq
\mathbb{E}
\int_0^T
\|u_0(t)\|^2\,dt.
\label{4.8}
\end{equation}
In view of \eqref{4.7} and \eqref{4.8}, we deduce
\[
\liminf_{k\to\infty}
\mathbb{E}
\int_0^T
\|u_k^*(t)\|^2\,dt
=
\mathbb{E}
\int_0^T
\|u_0(t)\|^2\,dt.
\]
Passing to a further subsequence if necessary, we conclude that
\begin{equation}
\mathbb{E}
\int_0^T
\|u_k^*(t)\|^2\,dt
\to
\mathbb{E}
\int_0^T
\|u_0(t)\|^2\,dt,
\qquad k\to\infty.
\label{4.9}
\end{equation}
In conjunction with \eqref{4.5}, this leads to
\begin{equation}
u_k^*
\longrightarrow
u_0
\quad\text{strongly in }L^2(\Omega_T;H),
\qquad k\to\infty.
\label{4.10}
\end{equation}
The latter establishes the strong compactness of the family
$\{u_k^*\}$. Moreover,
\[
J_k(u_k^*)
\leq
J_k(u^*).
\]
Hence, by Lemma \ref{lem:3.3}, we have
\begin{align}
\liminf_{k\to\infty}J_k(u_k^*)
&\leq
\liminf_{k\to\infty}
\left[
\mathbb{E}\int_0^T
\|y_k(t,u^*)\|^2\,dt
+
\mathbb{E}\int_0^T
\|u^*(t)\|^2\,dt
\right]
\nonumber\\
&=
\mathbb{E}\int_0^T
\|y(t,u^*)\|^2\,dt
+
\mathbb{E}\int_0^T
\|u^*(t)\|^2\,dt
=
J^*.
\label{4.11}
\end{align}
On the other hand, we have
\begin{align}
\liminf_{k\to\infty}J_k(u_k^*)
&=
\liminf_{k\to\infty}
\left[
\mathbb{E}\int_0^T
\|y_k^*(t,u_k^*)\|^2\,dt
+
\mathbb{E}\int_0^T
\|u_k^*(t)\|^2\,dt
\right]
\nonumber\\
&\geq
\mathbb{E}\int_0^T
\|y(t,u_0)\|^2\,dt
+
\mathbb{E}\int_0^T
\|u_0(t)\|^2\,dt
=
J(u_0).
\label{4.12}
\end{align}
Combining \eqref{4.11} and \eqref{4.12} we deduce that $u_0$ is
an optimal control. Therefore, the second assertion of
Theorem \ref{Th:2.3} is proved.

To prove the third assertion, it remains to show that
\begin{equation}
y_k^*(\cdot,u_k^*)
\to
y(\cdot,u_0)
\quad\text{strongly in }L^2(\Omega_T;H),
\qquad k\to\infty.
\label{4.13}
\end{equation}
Indeed, we have
\begin{align}
\left\|
y_k^*(t,u_k^*)-y(t,u_0)
\right\|
\leq{}&
\left\|
y_k(t,u_k^*)-y_k(t,u_0)
\right\|
\nonumber\\
&+
\left\|
y_k(t,u_0)-y(t,u_0)
\right\|.
\label{4.14}
\end{align}

The second term on the right-hand side of \eqref{4.14} converges to
zero in $L^2(\Omega_T;H)$ by Lemma \ref{lem:3.3}. Next, similarly to \eqref{3.32}, we have  
\[
\mathbb{E}
\left\|
y_k(t,u_k^*)-y_k(t,u_0)
\right\|^2
\leq
C_{19}
\int_0^T
\mathbb{E}
\left\|
u_k^*(t)-u_0(t)
\right\|^2\,dt
\to 0,
\qquad k\to\infty.
\]
where the last convergence follows from \eqref{4.10}. 
Thus \eqref{4.13} follows, which proves the third assertion of Theorem \ref{Th:2.3}.
Finally, to prove the first assertion of the theorem, we observe that, by
\eqref{4.10} and \eqref{4.13},
\[
\begin{aligned}
\lim_{k\to\infty}J_k^*
&=
\lim_{k\to\infty}
\mathbb{E}\int_0^T
\left\|
y_k^*(t,u_k^*)
\right\|^2\,dt
+
\lim_{k\to\infty}
\mathbb{E}\int_0^T
\left\|
u_k^*(t)
\right\|^2\,dt
\\
&=
\mathbb{E}\int_0^T
\left\|
y(t,u_0)
\right\|^2\,dt
+
\mathbb{E}\int_0^T
\left\|
u_0(t)
\right\|^2\,dt =
J(u_0)
=
J^*,
\end{aligned}
\]
which concludes the proof of Theorem \ref{Th:2.3}.
\end{proof}
\begin{proof}[Proof of Theorem \ref{Th:2.4}]
Recall that $U_{\mathrm{ad}}$ denotes the set of all admissible controls
for the problem \eqref{1.1}--\eqref{1.3}. Let
\[
J^*
=
\inf_{u\in U_{\mathrm{ad}}}J(u).
\] 
In view of \eqref{2.10} and
\eqref{3.35},
\[
J(0)
=
\mathbb{E}\int_0^\infty
e^{-\gamma t}
\|y(t,0)\|^2\,dt
<
\infty,
\]
which yields that $0 \in U_{\mathrm{ad}}$. Next, let $\{u_n\}$ be a minimizing sequence for $J$, such that
\begin{equation}
J(u_n)
\to
J^*,
\qquad n\to\infty.
\label{4.18}
\end{equation}
Thus, for sufficiently large $n$,
\begin{equation}
\begin{aligned}
J(u_n)
={}&
\mathbb{E}\int_0^\infty
e^{-\gamma t}
\|y(t,u_n)\|^2\,dt
+
\mathbb{E}\int_0^\infty
\|u_n(t)\|^2\,dt 
\leq
J^*+1.
\end{aligned}
\label{4.19}
\end{equation}
This implies that $\{u_n\}$ is weakly compact in
$L^2(\Omega_T;H)$, and
$\{y(\cdot,u_n)\}$ is weakly compact in
$L_\gamma^2(\Omega_\infty;H)$. Hence, passing to a subsequence if
necessary, we obtain
\begin{equation}
u_n
\rightharpoonup
u
\quad\text{weakly in }
L^2(\Omega_T;H),
\end{equation}
and
\begin{equation}
y(\cdot,u_n)
\rightharpoonup
z(\cdot)
\quad\text{weakly in }
L_\delta^2(\Omega_\infty;H),
\qquad n\to\infty.
\label{4.20}
\end{equation}
Now we show that there exists a probability space
$(\widetilde{\Omega},\widetilde{\mathcal{F}},\widetilde{P})$ such that
\[
\mathcal{L}\bigl(y(\cdot,u_n),u_n\bigr)
=
\mathcal{L}\bigl(\widetilde{y}(\cdot,\widetilde{u}_n),
\widetilde{u}_n\bigr),
\]
\[
\mathcal{L}(z)=\mathcal{L}(\widetilde{z}),
\qquad
\mathcal{L}(u)=\mathcal{L}(\widetilde{u}),
\]
and
\[
\widetilde{z}(t)
=
\widetilde{y}(t,\widetilde{u}).
\]
Indeed, by Lemma \ref{lem:3.2}, for every $N>0$, on $[0,N]$ we have
\begin{equation}
\widetilde{y}(\cdot,\widetilde{u}_n)
\rightharpoonup
\widetilde{y}(\cdot,\widetilde{u})
\quad\text{weakly in }
L^2(\widetilde{\Omega}_N;H),
\qquad n\to\infty.
\end{equation}
Note that the space $(\widetilde{\Omega}, \widetilde{\mathcal{F}}, \widetilde{P})$ may be chosen to be $([0,1], \mathcal{B}, \lambda)$, where $\mathcal{B}$ are Borel subsets of $[0,1]$, and $\lambda$ is Lebesgue measure. Thus, for any
$\varphi\in L_\gamma^2(\Omega_\infty;H)$, we have
\begin{align}
&\left|
\widetilde{\mathbb{E}}
\int_0^\infty
\int_D
\bigl(
\widetilde{y}(t,\widetilde{u}_n)
-
\widetilde{y}(t,\widetilde{u})
\bigr)
\varphi(t,x,\omega)\,dx\,
e^{-\gamma t}\,dt
\right|
\nonumber\\
&\qquad\leq
\left|
\widetilde{\mathbb{E}}
\int_0^N
\int_D
\bigl(
\widetilde{y}(t,\widetilde{u}_n)
-
\widetilde{y}(t,\widetilde{u})
\bigr)
\varphi(t,x,\omega)\,dx\,
e^{-\gamma t}\,dt
\right|
\nonumber\\
&\qquad\quad+
\widetilde{\mathbb{E}}
\int_N^\infty
\int_D
\left|
\widetilde{y}(t,\widetilde{u}_n)
-
\widetilde{y}(t,\widetilde{u})
\right|
\left|
\varphi(t,x,\omega)
\right|\,dx\,
e^{-\gamma t}\,dt
=: I_1+I_2.
\label{4.21}
\end{align}
Since $\varphi\in L_\delta^2(\widetilde{\Omega}_N;H)$,
for every fixed $N$,
\begin{equation}
I_1 \to 0,
\qquad n\to\infty.
\label{4.22}
\end{equation}
As far as the second term in \eqref{4.21}, for some constant $C>0$, we
obtain
\[
I_2
\leq
C
\left(
\int_N^\infty
\left[
\widetilde{\mathbb{E}}
\left\|
\widetilde{y}(t,\widetilde{u}_n)
\right\|^2
+
\widetilde{\mathbb{E}}
\left\|
\widetilde{y}(t,\widetilde{u})
\right\|^2
\right]
e^{-\gamma t}\,dt
\right)^{1/2}
\left(
\int_N^\infty
\widetilde{\mathbb{E}}
\|\varphi(t)\|^2
e^{-\gamma t}\,dt
\right)^{1/2}.
\]
Since $\widetilde{u}_n$ is bounded in
$L^2(\widetilde{\Omega}_\infty;H)$, in conjunction with \eqref{3.35}, we deduce that the first factor in the
right-hand side is uniformly bounded with respect to both $N$ and $n$,
whereas the second factor tends to zero as $N\to\infty$. Thus, for
every $\varepsilon>0$, we may choose $N>0$ sufficiently large so that
\[
I_2<\frac{\varepsilon}{2}
\]
uniformly with respect to $n$. Then, for this fixed $N$, choose
$n$ sufficiently large so that
\[
I_1<\frac{\varepsilon}{2}.
\]
Thus, \eqref{4.20} follows. Finally, using \eqref{4.18} and weak lower semicontinuity of the cost
functional, we obtain
\[
\begin{aligned}
J^*
&=
\lim_{n\to\infty}J(u_n) \geq
\liminf_{n\to\infty}
\widetilde{\mathbb{E}}
\int_0^\infty
\left\|
\widetilde{y}(t,\widetilde{u}_n)
\right\|^2
e^{-\gamma t}\,dt
+
\liminf_{n\to\infty}
\widetilde{\mathbb{E}}
\int_0^\infty
\left\|
\widetilde{u}_n(t)
\right\|^2\,dt
\\
&\geq
\widetilde{\mathbb{E}}
\int_0^\infty
\left\|
\widetilde{y}(t,\widetilde{u})
\right\|^2
e^{-\gamma t}\,dt
+
\widetilde{\mathbb{E}}
\int_0^\infty
\left\|
\widetilde{u}(t)
\right\|^2\,dt.
\end{aligned}
\]
This concludes the proof of Theorem \ref{Th:2.4}.
\end{proof}

\begin{proof}[Proof of Theorem \ref{Th:2.5}] Once again, to simplify the notation, we drop the superscripts $\sim$. Clearly,  
\begin{align}
J^*
&\leq
J(u_{T,\infty}^*) =
\mathbb{E}
\int_0^T
\left\|
y(t,u_{T,\infty}^*)
\right\|^2
e^{-\gamma t}\,dt
+
\mathbb{E}
\int_0^T
\left\|
u_{T,\infty}^*(t)
\right\|^2\,dt
\nonumber\\
&\quad+
\mathbb{E}
\int_T^\infty
\left\|
y(t,u_{T,\infty}^*)
\right\|^2
e^{-\gamma t}\,dt =
J_T^*+\varphi(T).
\label{4.23}
\end{align}
By \eqref{3.35},  we have $\varphi(T) \to 0$ as $T \to \infty$. On the other hand, we have
\begin{align}
J^*
=
J(u^*)
={}&
\mathbb{E}
\int_0^T
\left\|
y(t,u^*)
\right\|^2
e^{-\gamma t}\,dt
+
\mathbb{E}
\int_0^T
\left\|
u^*(t)
\right\|^2\,dt
\nonumber\\
&+
\mathbb{E}
\int_T^\infty
\left\|
y(t,u^*)
\right\|^2
e^{-\gamma t}\,dt
+
\mathbb{E}
\int_T^\infty
\left\|
u^*(t)
\right\|^2\,dt
\geq
J_T^*+\psi(T).
\label{4.24}
\end{align}
Again, $\psi(T)\to 0$ as $T\to\infty$. Therefore, from
\eqref{4.23} and \eqref{4.24}, we obtain
\[
J^*-J_T^*
\leq
\varphi(T)
\to 0,
\qquad T\to\infty,
\]
and
\[
J^*-J_T^*
\geq
\psi(T)
\to 0,
\qquad T\to\infty.
\]
Hence, the first assertion of Theorem \ref{Th:2.5} holds. Let us prove the second assertion. Set $T=n$. We have
\[
J_n^*
=
J_n(u_n^*)
\longrightarrow
J^*,
\qquad n\to\infty.
\]
Hence, there exists a constant $C>0$ such that
\[
\begin{aligned}
J_n(u_n^*)
&=
\mathbb{E}
\int_0^n
\|y^*(t,u_n^*)\|^2
e^{-\gamma t}\,dt
+
\mathbb{E}
\int_0^n
\|u_n^*(t)\|^2\,dt
\\
&=
\mathbb{E}
\int_0^n
\|y^*(t,u_n^*)\|^2
e^{-\gamma t}\,dt
+
\mathbb{E}
\int_0^\infty
\|u_{n,\infty}^*(t)\|^2\,dt
\leq C.
\end{aligned}
\]
Thus, $\{u_{n,\infty}^*\}$ is weakly compact in
$L^2(\Omega_\infty;H)$. Therefore, passing to a subsequence if
necessary,
\[
u_{n,\infty}^*
\rightharpoonup
u_\infty
\quad\text{weakly in }
L^2(\Omega_\infty;H),
\qquad n\to\infty.
\]
Furthermore, 
\[
\begin{aligned}
J(u_{n,\infty}^*)
&=
\mathbb{E}
\int_0^\infty
\|y(t,u_{n,\infty}^*)\|^2
e^{-\gamma t}\,dt
+
\mathbb{E}
\int_0^\infty
\|u_{n,\infty}^*(t)\|^2\,dt
\\
&=
J_n^*
+
\mathbb{E}
\int_n^\infty
\|y(t,u_{n,\infty}^*)\|^2
e^{-\gamma t}\,dt
\to
J^*,
\qquad n\to\infty.
\end{aligned}
\]
This proves the second assertion of the theorem. Proceeding with the proofs of the third and the fourth assertions, let $y(t,u_\infty)$ be a weak martingale solution corresponding to
$u_\infty$. We now show that the pair
\[
\bigl(u_\infty,y(t,u_\infty)\bigr)
\]
is optimal. Indeed, by  weak lower semicontinuity,
\begin{align}
J^*
&=
\lim_{n\to\infty}
J(u_{n,\infty}^*) \geq
\liminf_{n\to\infty}
\mathbb{E}
\int_0^\infty
\|y(t,u_{n,\infty}^*)\|^2
e^{-\gamma t}\,dt
+
\liminf_{n\to\infty}
\mathbb{E}
\int_0^\infty
\|u_{n,\infty}^*(t)\|^2\,dt
\nonumber\\
&\geq
\mathbb{E}
\int_0^\infty
\|y(t,u_\infty)\|^2
e^{-\gamma t}\,dt
+
\mathbb{E}
\int_0^\infty
\|u_\infty(t)\|^2\,dt.
\label{4.25}
\end{align}
The proof of Theorem \ref{Th:2.5} is now complete.
\end{proof}

\begin{corollary}
Under the assumptions of Theorem \ref{Th:2.5}, we have, along a subsequence,
\[
u_{n,\infty}^*
\to
u_\infty
\quad\text{ strongly in }
L^2(\Omega_\infty;H).
\]
\end{corollary}
\begin{proof}
Indeed, from \eqref{4.25}, we have
\[
\begin{aligned}
J^*
&=
\mathbb{E}
\int_0^\infty
\left\|
y(t,u_\infty)
\right\|^2
e^{-\gamma t}\,dt
+
\liminf_{n\to\infty}
\mathbb{E}
\int_0^\infty
\left\|
u_{n,\infty}^*(t)
\right\|^2\,dt
\\
&=
\mathbb{E}
\int_0^\infty
\left\|
y(t,u_\infty)
\right\|^2
e^{-\gamma t}\,dt
+
\mathbb{E}
\int_0^\infty
\left\|
u_\infty(t)
\right\|^2\,dt.
\end{aligned}
\]

Therefore, passing to a subsequence if necessary,
\[
\mathbb{E}
\int_0^\infty
\left\|
u_{n,\infty}^*(t)
\right\|^2\,dt
\to
\mathbb{E}
\int_0^\infty
\left\|
u_\infty(t)
\right\|^2\,dt.
\]

Hence,
\[
u_{n,\infty}^*
\to
u_\infty
\quad\text{strongly in }
L^2(\Omega_\infty;H).
\]
\end{proof}

\section*{Acknowledgments}  The research of Oleksandr Misiats was supported by Simons Collaboration
Grant for Mathematicians No. 854856 and National Science Foundation Grant DMS-2408507.

\end{document}